\documentclass[11pt,a4paper,oneside]{amsart}
\newcounter{samcomments}

\usepackage{amsmath} % Lots of maths functionality
\usepackage{amssymb} % Maths symbols
\usepackage{amsthm} % Maths environments: \begin{proof}, etc.
\usepackage{stmaryrd} % [[ brackets
\usepackage[english]{babel} % Language and hyphenation.
\usepackage[font=small,justification=centering]{caption} % More flexibility for captioning figures
\usepackage[nodayofweek]{datetime}
\usepackage[shortlabels]{enumitem} % Change enumeration labelling with \begin{enumerate}[a)] etc.
\usepackage[T1]{fontenc} % For font encoding, to allow accents, copy & paste, inequality signs, etc. to all work nicely.
\usepackage[utf8]{inputenc} % To be loaded after fontenc, also for encoding.
\usepackage{ifthen} % For Dani's \begin{com} environment and \numberedtheorem
\usepackage{mathtools} % Uses amsmath, fixes quirks and adds functionality.
\usepackage[dvipsnames]{xcolor} % Allow links to have colour. Needs to be before hyperref and tikz-cd
\usepackage[pdftex,  colorlinks=true]{hyperref} % Makes references and citations into links
    \hypersetup{urlcolor=RoyalBlue, linkcolor=RoyalBlue,  citecolor=black}
\usepackage{setspace} % Allows \onehalfspacing etc. for changing gaps between lines
\usepackage{tikz-cd} % Commutative diagrams
\usepackage{xfrac} % Nicer quotients 
\usepackage[capitalize]{cleveref} % use \Cref{} for instead of X~\ref{} 
\makeatletter
\@namedef{subjclassname@1991}{Mathematical subject classification 1991}
\@namedef{subjclassname@2000}{Mathematical subject classification 2000}
\@namedef{subjclassname@2010}{Mathematical subject classification 2010}
\@namedef{subjclassname@2020}{Mathematical subject classification 2020}
\makeatother

\newtheorem{thm}{Theorem}[section]
\newtheorem{lemma}[thm]{Lemma}

\theoremstyle{definition}

\newtheorem{example}[thm]{Example}

\theoremstyle{plain}

    \newtheoremstyle{TheoremNum}
        {\topsep}{\topsep} %%% space between body and thm
        {\itshape} %%% Thm body font
        {-0.25cm} %%% Indent amount (empty = no indent)
        {\bfseries} %%% Thm head font
        {.} %%% Punctuation after thm head
        { }  %%% Space after thm head
        {\thmname{#1}\thmnote{ \bfseries #3}}%%% Thm head spec
    \theoremstyle{TheoremNum}

\newcommand*{\claimproofname}{My proof}

\DeclareMathOperator{\Ker}{\mathrm{Ker}}
\DeclareMathOperator{\coker}{\mathrm{Coker}}

\DeclareMathOperator{\Tor}{\mathrm{Tor}}

\newcommand{\GL}{\mathrm{GL}}

\newcommand{\SL}{\mathrm{SL}}

\newcommand{\SO}{\mathrm{SO}}

\DeclareMathOperator{\id}{id}

\newcommand{\Stn}{\mathrm{St}_N}

\def\Z{\mathbb{Z}}

\newcommand{\ZZ}{\mathbb{Z}}

\newcommand{\RR}{\mathbb{R}}
\newcommand{\QQ}{\mathbb{Q}}

\usepackage{tikz}
\usetikzlibrary{arrows,quotes}
\tikzstyle{blackNode}=[fill=black, draw=black, shape=circle]

\newcommand{\group}{G} %G or \Gamma seem the reasonable choices
\newcommand{\rankSL}{N}
\newcommand{\homdeg}{n}

\newcommand{\Stab}{\operatorname{Stab}}
\newcommand{\Stabcoset}{\group}
\newcommand{\orbitrep}{\mathcal O}

\newcommand{\ls}{\left\{}
\newcommand{\rs}{\right\}}
\newcommand{\ms}{\, \middle| \,}

\title[Unitary cohomology of $\SL_{\rankSL}(\Z)$]{Non-vanishing unitary cohomology of low-rank integral special linear groups}

\usepackage[foot]{amsaddr}

\usepackage{listofitems}
\newcommand\cycle[2][\,]{%
	\readlist\thecycle{#2}%
	(\foreachitem\i\in\thecycle{\ifnum\icnt=1\else#1\fi\i})%
}

\newcounter{dawidcomments}
\newcommand{\dawid}[1]{\textbf{\color{red}(D\arabic{dawidcomments})}
\marginpar{\scriptsize\raggedright\textbf{\color{red}(D\arabic{dawidcomments})Dawid: }#1}
\addtocounter{dawidcomments}{1}}

\newcounter{benjamincomments}
\newcommand{\benjamin}[1]{\textbf{\color{red}(B\arabic{benjamincomments})}
\marginpar{\scriptsize\raggedright\textbf{\color{red}(B\arabic{benjamincomments})Benjamin: }#1}
\addtocounter{benjamincomments}{1}}

\newcounter{piotrcomments}

\author{Benjamin Br\"uck}
\author{Sam Hughes}
\author{Dawid Kielak}
\author{Piotr Mizerka}
\address[B.~Br\"uck]{Institut f{\"u}r Mathematische Logik und Grundlagenforschung, Einsteinstr. 62, 48149 M{\=u}nster, Germany}
\address[S.~Hughes]{Universit\"at Bonn, Mathematical Institute, Endenicher Allee 60, 53115 Bonn, Germany}
\address[D.~Kielak]{Mathematical Institute, Andrew Wiles Building, Observatory Quarter, University of Oxford, Oxford OX2 6GG, UK}
\address[P.~Mizerka]{Insitute of Mathematics, Polish Academy of Sciences, Śniadeckich 8, 00-656 Warszawa, Poland}
\email{benjamin.brueck@uni-muenster.de}
\email{sam.hughes.maths@gmail.com}
\email{kielak@maths.ox.ac.uk}
\email{pmizerka@impan.pl}
\date{\today}

\begin{document}
	\sloppy
\maketitle

\begin{abstract}
We construct explicit finite-dimensional orthogonal representations $\pi_\rankSL$ of $\SL_{\rankSL}(\Z)$ for ${\rankSL} \in \{3,4\}$  all of whose invariant vectors are trivial, and such that $H^{\rankSL - 1}(\SL_{\rankSL}(\Z),\pi_\rankSL)$ is non-trivial. This implies that for $\rankSL$ as above, the group $\SL_{\rankSL}(\Z)$ does not have property $(T_{\rankSL-1})$ of Bader--Sauer and therefore is not $(\rankSL-1)$-Kazhdan in the sense of De Chiffre--Glebsky--Lubotzky--Thom, both being higher versions of Kazhdan's property $T$.
\end{abstract}

\section{Introduction}

One of the remarkable qualities of Kazhdan's property $T$ is that it admits a plethora of equivalent formulations. In particular, the celebrated Delorme--Guichardet Theorem \cite{Guichardet1972,Delorme1977} tells us that a finitely generated group has property $T$ if and only if its first reduced cohomology with coefficients in any unitary representation is zero.

 This cohomological viewpoint  invites natural extensions, where one looks at the vanishing of higher reduced cohomologies. Bader--Sauer \cite{BaderSauer2023} introduced two such generalisations: the weaker property $(T_n)$, that requires the $n$th reduced cohomology to vanish when the coefficients come from a unitary representation all of whose invariant vectors are trivial, and the stronger property $[T_n]$, where the vanishing should happen for all unitary representations; the latter property is equivalent to being $n$-Kazhdan, as introduced by De Chiffre--Glebsky--Lubotzky--Thom \cite{DeChiffreetal2020}. 
 
 Kazhdan introduced property $T$ to study lattices in semi-simple Lie groups of higher rank \cite{Kazhdan1967}. The prime example of such a lattice is $\SL_{\rankSL}(\Z)$ for $\rankSL \geqslant 3$, and it is precisely these groups that we will investigate.
 
 \begin{thm}
 \label{thm:main_intro}
 	For every $\rankSL \in \{3,4\}$, there exists a finite-dimensional orthogonal representation $\pi_\rankSL$ of $\SL_{\rankSL}(\Z)$  all of whose invariant vectors are trivial such that
 \[
 H^{\rankSL -1}(\SL_{\rankSL}(\Z); \pi_\rankSL) \neq 0.
 \] 
 Since $\pi_\rankSL$ is finite-dimensional, the reduced and non-reduced cohomologies are equal, and therefore $\SL_{\rankSL}(\Z)$ does not have property $(T_{\rankSL -1})$.
 \end{thm}

We remark that work by Monod \cite[Corollary 1.6]{Monod2010} implies that the non-trivial cohomology classes we obtain are unbounded.

\cref{thm:main_intro} should be compared with \cite[Theorem A]{BaderSauer2023}, stating that for all $\rankSL \geqslant 3$, the groups $\SL_{\rankSL}(\Z)$ have property $(T_{\rankSL -2})$. The Bader--Sauer theorem is an example of the phenomenon of cohomology vanishing below the rank (which is $\rankSL -1$ in this case); our result shows that at the rank, such vanishing no longer takes place.

For $\rankSL = 2$, since $\SL_{2}(\Z)$ has a finite-index subgroup with infinite abelianisation, one easily constructs unitary representations with all fixed vectors trivial that admit non-trivial harmonic cocycles, and hence the first cohomology of $\SL_{2}(\Z)$ with coefficients in such a representation is non-trivial. Hence, $\SL_{2}(\Z)$ does not have property $(T_1)$.

There is an easier way of establishing that $\SL_{3}(\Z)$ does not have the stronger property $[T_3]$, shown to the authors by Roman Sauer: $\SL_{3}(\Z)$ admits  a finite-index torsion-free subgroup with cohomological dimension three and Euler characteristic zero. Since the zeroth cohomology of a non-trivial group with trivial coefficients $\QQ$ is non-zero, there must be some non-trivial cohomology in odd dimensions. There is none in dimension one, since $\SL_{3}(\Z)$ has property $T$, and thus its finite-index subgroups have finite abelianisations. The subgroup must therefore have non-vanishing third cohomology with coefficients in $\QQ$, which by Shapiro's lemma gives us non-vanishing third cohomology for $\SL_{3}(\Z)$ with coefficients in a finite-dimensional unitary representation.  This representation does have non-trivial invariant vectors.
 
\medskip

The general strategy that we will follow consists of three steps. First, we explicitly construct a chain complex for the symmetric space of $\SL_{\rankSL}(\RR)$ relative to its Borel--Serre boundary  using Voronoi cells -- here we are following an established technique, see \cite{Soule2000,ElbazVincent2013}. Then we construct an explicit finite-dimensional representation of $\SL_{\rankSL}(\Z)$  all of whose invariant vectors are trivial. Finally, we tensor the chain complex with the representation, and obtain non-trivial homology classes of the tensored complex using a computer. Through an argument using a spectral sequence and Borel--Serre duality, we obtain non-trivial cohomology classes for $\SL_{\rankSL}(\Z)$ with coefficients in the chosen representation.

Notebooks containing the computations can be found in a Zenodo repository \cite{HigherTSLN2024}. In particular, they contain an implementation of the Voronoi tessellation in the Julia language.

\section{Computing cohomology of special linear groups}

The central aim of the article is to compute cohomology groups of
the special linear group $\SL_{\rankSL}(\Z)$. We will do this by relating the cohomology of $\SL_{\rankSL}(\Z)$ to the homology of the pair $(X_{\rankSL}^\ast,\partial X_{\rankSL}^\ast)$, where $X_\rankSL$ denotes the symmetric space associated to $\SL_\rankSL(\RR)$, the space $X_{\rankSL}^\ast\supset X_\rankSL$ is a certain bordification, and $\partial X_{\rankSL}^\ast = X_{\rankSL}^\ast \setminus X_{\rankSL}$ is the boundary of this bordification.

\subsection{The bordification $X_\rankSL^*$}
\label{sec_bordification}
We start by defining $X_\rankSL^*$.
The set of all symmetric $\rankSL\times \rankSL$-matrices over $\RR$ forms an $(\rankSL(\rankSL+1)/2)$-dimensional subspace of $\RR^{\rankSL \times \rankSL}$. We identify it with the subspace of quadratic forms on $\RR^\rankSL$.
Inside this subspace, the set of all positive definite forms is a cone that we denote by $K_\rankSL$. We write $K_\rankSL^*$ for its rational closure, i.e.~the set of all positive semidefinite forms whose kernel is spanned by vectors in $\QQ^\rankSL$. The set $K_\rankSL^*$ forms a cone as well and we have $K_\rankSL\subset K_\rankSL^*$.

Define  $X_{\rankSL}^*$ as the quotient of $K_\rankSL^*$ by homotheties and let $\pi\colon K_\rankSL^* \to X_{\rankSL}^*$ be the projection map, i.e.~$\pi(q) = \pi(q')$ if and only if $q = \lambda\cdot q'$ for some $\lambda\in \RR_{>0}$.
We identify the symmetric space $X_\rankSL$ associated to $\SL_{\rankSL}(\RR)$ with $\pi(K_\rankSL)$. (The isomorphism of $X_\rankSL$  with the coset description of the symmetric space as $\SO(\rankSL) \backslash \SL_\rankSL(\RR)$ is given by $\SO(\rankSL) g \mapsto \pi(g^t g)$.)
We write $\partial X_{\rankSL}^* = X_{\rankSL}^* \setminus X_{\rankSL}$.

The group $\SL_\rankSL(\ZZ)$ acts on $K_{\rankSL}^*$ from the right by 
\begin{equation*}
	q . g = g^t q g, \text{ for } g\in \SL_\rankSL(\ZZ) \text{ and } q\in K_{\rankSL}^*,
\end{equation*}
where we see both $g$ and $q$ as represented by matrices in $\RR^{\rankSL\times \rankSL}$.
This induces an action of $\SL_\rankSL(\ZZ)$ on $X_{\rankSL}^*$ that extends the usual action on $X_{\rankSL}$ and in particular preserves $\partial X_{\rankSL}^*$.

\subsection{Relation to the cohomology of $\SL_\rankSL(\ZZ)$}

Write $G = \SL_{\rankSL}(\Z)$ and let $M$ be a $\QQ G$-module which is finite dimensional as a $\QQ$-module.
By \cite[Proposition 1]{Soule2000},
\[H_p(X_{\rankSL}^\ast,\partial X_{\rankSL}^\ast) = \left\{ \begin{tabular}{cl}
	$\Stn$    & if $p={\rankSL}-1$; \\
	$0$       & \text{otherwise};
\end{tabular}\right.  \]
as $\group$-modules, where $\Stn$ is the \emph{Steinberg module} associated to $\SL_{\rankSL}(\QQ)$, i.e.~the degree-$(\rankSL-2)$ reduced homology of the Tits building associated to $\SL_{\rankSL}(\QQ)$.  In particular,
\begin{equation}\label{Stn tensor M}
	H_{p}(X_{\rankSL}^\ast,\partial X_{\rankSL}^\ast)\otimes M = \left\{ \begin{tabular}{cc}
		$\Stn \otimes M$    & if $p={\rankSL}-1$; \\
		$0$       & \text{otherwise};
	\end{tabular}\right.
\end{equation}
as $\group$-modules, where tensoring takes place over $\Z$, and $G$ acts diagonally.

We want to compare the $\group$-modules $H_p(X_{\rankSL}^\ast,\partial X_{\rankSL}^\ast) \otimes M$ and $H^\group_p(X_{\rankSL}^\ast,\partial X_{\rankSL}^\ast;M)$. To this end, let us prove the following lemma, which holds for any group $\group$.

\begin{lemma}\label{lem:bilinear_equivariant}
	Let $(X,Y)$ be a pair of $\group$-CW complexes.
	If $M$ is a $G$-module that is torsion free as a $\Z$-module, then the natural bilinear map $H_p(X,Y;M)\to H_p(X,Y)\otimes M$ is an isomorphism of $G$-modules.
\end{lemma}
\begin{proof}
	We proceed by examining a standard proof for the Universal Coefficient Theorem and checking that each of the maps involved is in fact a $\group$-map.  See for example \cite[Chapter~3.A]{Hatcher2002} for a detailed proof of the Universal Coefficient Theorem.
	
	Let $C_\bullet$ denote the chain complex of the pair $(X,Y)$ -- it is a chain complex of $\Z \group$-modules, and all the modules are free as $\Z$-modules. The complex admits subcomplexes $Z_\bullet$ and $B_{\bullet}$ (with trivial differentials) of $\group$-modules consisting of cycles and boundaries, respectively.
	The chain complex $C_\bullet$ is an extension of $B_{\bullet-1}$ by $Z_\bullet$; this extension respects the $\group$-module structure, and it is split as an extension of chain complexes of $\Z$-modules.
	Tensoring these chain complexes with $M$ over $\Z$ (with diagonal $\group$-action) we obtain a short exact sequence of chain complexes of $\group$-modules
	\[\begin{tikzcd}
		0 \arrow[r] & Z_\bullet \otimes M \arrow[r] & C_\bullet \otimes M \arrow[r] & B_{\bullet-1} \otimes M \arrow[r] & 0
	\end{tikzcd}\]
(exact since the short exact sequence before tensoring was split).
	It gives a long exact homology sequence of $\group$-modules
	\[\begin{tikzcd}
		\cdots \arrow[r] & B_n\otimes M \arrow[r,"i_n\otimes \id"] & Z_n\otimes M \arrow[r] & H_n(C_\bullet;M) \arrow[r] & \cdots
	\end{tikzcd}\]
	that breaks up into short exact sequences of $\group$-modules
	\[\begin{tikzcd}
		0 \arrow[r] & \coker(i_n\otimes\id) \arrow[r] & H_n(C_\bullet;M) \arrow[r] & \ker(i_{n-1}\otimes\id) \arrow[r] & 0.
	\end{tikzcd} \]
	Now, $\coker(i_n\otimes\id)=H_n(C_\bullet)\otimes_\Z M$ by right-exactness of the tensor product.  The group $\ker(i_{n-1}\otimes\id)$ is, by definition (see e.g. \cite[Chapter 3.A]{Hatcher2002}), exactly $\Tor_1^\Z(H_{n-1}(C_\bullet),M)$. The $\Tor$-group vanishes because $M$ is $\Z$-torsion-free by assumption.
\end{proof}

Returning to the case $G = \SL_{\rankSL}(\Z)$, we conclude that 
\begin{equation}
	\label{UCT}
	H_p(X_{\rankSL}^\ast,\partial X_{\rankSL}^\ast) \otimes M \cong H_p(X_{\rankSL}^\ast,\partial X_{\rankSL}^\ast;M)
\end{equation} as $\group$-modules.

\iffalse
By \eqref{eqn SS1} there is a spectral sequence
\[E^1_{p,q}=\bigoplus_{\sigma\in\orbitrep_p} H^q(\group_\sigma;M_\sigma)\Rightarrow H^\group_{p+q}(X_{\rankSL}^\ast,\partial X_{\rankSL}^\ast;M). \]
\dawid{how is this used? Or is it used below?}
\benjamin{I think this is a question for Sam.}
\fi

There is a spectral sequence, see \cite[VII (7.2)]{Brown1982},  that computes the equivariant homology $H^\group_{p+q}(X_{\rankSL}^\ast,\partial X_{\rankSL}^\ast;M)$, namely,
\[ E^2_{p,q}=H_p(\group;H_q(X_{\rankSL}^\ast,\partial X_{\rankSL}^\ast;M))\Rightarrow H^\group_{p+q}(X_{\rankSL}^\ast,\partial X_{\rankSL}^\ast;M).\]
But  by \cref{Stn tensor M,UCT}, the $E^2$-page of this spectral sequence is concentrated in the $q={\rankSL}-1$ row.  In particular, it collapses, and so
\begin{align*}
	\begin{split}
		H^\group_{p+N-1}(X_{\rankSL}^\ast,\partial X_{\rankSL}^\ast;M) &\cong H_{p}(\group;H_{{\rankSL}-1}(X_{\rankSL}^\ast,\partial X_{\rankSL}^\ast;M)) \\
		&\cong H_{p}(\group;\Stn \otimes M) \\
		&\cong H^{ {\rankSL}(\rankSL - 1)/2 - p }(\group; M)
	\end{split}
\end{align*}
where the last isomorphism is Borel--Serre Duality \cite{BorelSerre1973}.

In conclusion, in order to understand the cohomology $H^q(\group; M)$, it is enough to compute $H^\group_{(N+2)(N-1)/2 - q}(X_{\rankSL}^\ast,\partial X_{\rankSL}^\ast;M)$.
We are in particular interested in the case $q=\rankSL-1$, where the above gives an isomorphism
\begin{equation}
\label{eqn:SpecSeq}
	H^{\rankSL-1}(\group; M) \cong H^\group_{\frac{N(N-1)}{2}}(X_{\rankSL}^\ast,\partial X_{\rankSL}^\ast;M).
\end{equation}
We will compute the right hand side of this equation using an explicit chain complex that we describe in the next subsection.

\section{A chain complex for $\left(X_{\rankSL}^*,\partial X_{\rankSL}^*\right)$}

In this section, we describe a cell structure on $(X_\rankSL^*,\partial X_\rankSL^*)$ due to Voronoi \cite{Voronoi1908}. This cell complex (or its quotient under the action of $\SL_\rankSL(\ZZ)$) is often called the \emph{first Voronoi} or \emph{perfect cone} decomposition of $X_\rankSL^*$. We then give an explicit description of the associated cellular chain complex.
We largely follow \cite{Soule2000} and \cite{ElbazVincent2013}; for further details, see also \cite[Chapter 7]{Martinet2003} and \cite[Sections 2.7--2.10]{McConnell1991}

\subsection{Voronoi tessellation}
\label{Voronoi sec}

Recall from \cref{sec_bordification} that $K_\rankSL$ and $K_\rankSL^*$ denote the cones of positive definite forms and positive semi-definite forms with rational kernel, respectively, and that $X_\rankSL$ and $X_\rankSL^*$ are the images of these cones under the map $\pi$ that quotients out homotheties.
For a positive definite form $q\in K_\rankSL$, let $\mu(q) = \min_{w\in \ZZ^\rankSL \setminus \{0\}}q(w)$ be the smallest value that $q$ takes on non-trivial elements of $\ZZ^\rankSL$. The set of \emph{minimal vectors} $m(q)$ is the (finite) subset of $\ZZ^\rankSL$ where this miminum is attained,
\begin{equation*}
	m(q) = \ls v \in \ZZ^\rankSL \ms q(v) = \mu(q) \rs.
\end{equation*}
The form $q$ is called \emph{perfect} if $m(q)$ determines it uniquely up to homothety, so if $q'\in K_\rankSL$ with $m(q') = m(q)$, then $\pi(q) = \pi(q')\in X_\rankSL$.

To each perfect form $q\in K_\rankSL$ we can associate a subset $\sigma(q)\subseteq X_\rankSL^*$ defined as follows: For $v \in \ZZ^{\rankSL} \setminus \{0 \}$, let $\hat{v}\in K_\rankSL^*$ be the positive semidefinite quadratic form defined by the matrix $vv^t$.  
The convex hull of all $\hat{v}$ with $v\in m(q)$ is a subset of $K_\rankSL^*$. Define $\sigma(q)$ as the image of this convex hull in $X_\rankSL^*$,
\begin{equation*}
	 \sigma(q) \coloneqq \pi\left(\operatorname{hull}(\ls \hat{v} \,\middle|\, v\in m(q) \rs)\right).
\end{equation*}
Voronoi \cite{Voronoi1908} showed that the collection of the sets $\sigma(q)$, where $q$ ranges over all perfect forms in $K_\rankSL$, together with all their intersections 
forms a polyhedral cell decomposition of $X_\rankSL^*$  \cite[Proposition 7.1.8]{Martinet2003}.

Any face $\tau$ of this polyhedral complex is contained in a maximal dimensional cell $\sigma(q)$. Define  $m(\tau)\subseteq m(q)$ to be the set of all minimal vectors $v\in m(q)$ such that $\pi(\hat{v})\in \tau$ (in particular, if $\tau = \sigma(q)$ has maximal dimension, then $m(\sigma(q)) = m(q)$). The set  $m(\tau)$ is uniquely determined by $\tau$ and independent of its embedding in a maximal cell $\tau\subseteq \sigma(q)$ \cite[Theorem 2.10(a)]{McConnell1991}. The cell $\tau$ is the convex hull of $m(\tau)$, and for cells $\tau$ and $\tau'$ we have $m(\tau \cap \tau') = m(\tau) \cap m(\tau')$.

The action of $\SL_\rankSL(\ZZ)$ on $X_\rankSL^*$ is cellular with respect to this cell structure and for $g\in \SL_\rankSL(\ZZ)$ and a cell $\sigma$, we have $m(\sigma.g) = m(\sigma).g = \ls gv \ms v\in m(\sigma) \rs$.
There are only finitely many $\SL_\rankSL(\ZZ)$-orbits of cells \cite[p.~110]{Voronoi1908}, cf.~\cite[Theorem 2.10(c)]{McConnell1991}, and if
a cell $\sigma$ intersects $X_\rankSL$ non-trivially (so it is not contained in $\partial X_\rankSL^*$), then the setwise stabiliser $\Stab_{\SL_\rankSL(\ZZ)}(\sigma)$ is finite.

Note that for $\rankSL\geq 2$, all vertices of this polyhedral complex are contained in $\partial X_\rankSL^*$ (they correspond to the forms $vv^t$, which cannot be positive definite as they have non-trivial kernels).
Furthermore, if the interior of a cell $\tau$ contains any point of $\partial X_\rankSL^*$, then $\tau$ is already entirely contained in $\partial X_\rankSL^*$. In particular, the cell decomposition is such that $\partial X_\rankSL^*$ is a subcomplex (i.e.~a union of closed cells).
This allows one to compute $H^\group_{p}(X_{\rankSL}^\ast,\partial X_{\rankSL}^\ast;M)$ using the cellular chain complex of the pair $(X_\rankSL^*, \partial X_\rankSL^*)$. We describe this chain complex in the next section.

\begin{example}
If $\rankSL = 2$, the symmetric space $X_2$ is the hyperbolic plane and $X_2^*$ is obtained from it by adding a countable set of boundary points. The polyhedral complex described above has dimension 2 and is simplicial. There is exactly one $\SL_2(\ZZ)$-orbit of cells in each dimension 0,1 and 2. The 2- and 1-cells intersect $X_2$ non-trivially, whereas the 0-cells are contained in $\partial X_2^*$.
The orbit of the 2-cells is represented by the perfect form 
\begin{equation*}
q = \begin{pmatrix}
2 & -1 \\
-1 & 2
\end{pmatrix}
 \text{ with } m(q) = \ls \pm e_1, \pm e_2, \pm (e_1+e_2) \rs,
\end{equation*}
where $e_1, e_2\in \ZZ^2$ are the two standard basis vectors.
The orbit of 1-cells is represented by the cell $\sigma$ with $m(\sigma) = \ls \pm e_1, \pm e_2 \rs$, and the orbit of the 0-cells by $\tau$ with $m(\tau) = \ls \pm e_1\rs$.
This tessellates $X_2^*$ by the Farey graph with all vertices lying at $\partial X_2^*$.

The case $N=3$ is described in \cite[A.3.7]{Stein2007}.
\end{example}

\iffalse
\section{A spectral sequence}
A \emph{$\group$-complex} is a CW-complex on which $G$-acts cellularly \dawid{we allow orientations to be reversed. Is this a problem here?}. For any $\group$-module $M$ and pair of $\group$-complexes $(X,Y)$ there is a spectral sequence given by taking the union of $Y$ with a filtration of the skeleton of $X$, see \cite[Section 2.1]{Soule2000}.  Specifically we have,
\[E^1_{p,q}=H_q^\group(X^{(p)},X^{(p-1)}\cup Y;M)\Rightarrow H_q^\group(X,Y;M), \]
where $X^{(p)}$ denotes the $p$-skeleton of $X$.
By \cite[VII (8.1)]{Brown1982} (see also \cite[Proposition~2]{Soule2000}) there are natural isomorphisms
\begin{equation}\label{eqn SS1}
 E^1_{p,q}=\bigoplus_{\sigma\in\orbitrep_p} H_q(\group_\sigma;M_\sigma)
\end{equation}
where $\orbitrep_p$ is a set of representative orbits of $p$-cells of $X^{(p)}-(X^{(p-1)}\cup Y)$\benjamin{There seems to be redundancy here. Isn't it just "... $p$-cells of $X- Y$"?} and $M_\sigma=\Z_\sigma\otimes M$, where $\Z_\sigma$ denotes the abelian group $\Z$ on which $G$ acts by the orientation character.  The differential on the $E^1$-page is described in \cite[VII (8.1)]{Brown1982}.

\iffalse
There is a second spectral sequence (see \cite[VII (7.1)]{Brown1982}) given by
\begin{equation}\label{eqn SS2}
    E^2_{p,q}=H_p(\group;H_q(X,Y;M))\Rightarrow H^\group_{p+q}(X,Y;M).
\end{equation}

In addition to the spectral sequences above, we will also need the following lemma; it is certainly known to experts, however, we could not find a proof in the literature.
\fi

\fi

\subsection{The cellular chain complex of $\left(X_{\rankSL}^*,\partial X_{\rankSL}^*\right)$}\label{section:computing_cohomology}
\iffalse Let $1\leq \homdeg\leq d({\rankSL})$, where $d({\rankSL})$ \dawid{$=$... let's give the number} is the dimension of the complex $\left(X_{\rankSL}^*,\partial X_{\rankSL}^*\right)$, be given. The $\homdeg$-th reduced homology $H_\homdeg\left(X_{\rankSL}^*,\partial X_{\rankSL}^*\right)$ with coefficients given by a representation $\pi \colon \group\rightarrow\mathcal{B}(\mathcal{H})$ on a Hilbert space $\mathcal{H}$ can be computed in the following way. \dawid{I don't like the fact that the coefficients are not present in the notation} REFERENCE
\dawid{we should say here that our Hilbert spaces will be finite dimensional, and hence there will be no need for taking closures}
\fi
\label{sec_chain_complex_Voronoi}

We will now describe the cellular chain complex $V_\bullet$ of $(X_\rankSL^*, \partial X_\rankSL^*)$, as a chain complex of projective $\QQ\SL_\rankSL(\ZZ)$-modules.
The discussion is actually more general: consider $(X, \partial X)$, a $\group$-CW-pair of regular CW-complexes for some group $\group$, such that there are only  with finitely many orbits of open cells that are not contained in $\partial X$ and such that the stabiliser of each such cell is finite.

Given an $\homdeg$-cell $\sigma$ and an $(\homdeg -1)$-cell $\tau$, we say that $\tau$ is a \emph{facet} of $\sigma$ if and only if the attaching map $\partial \sigma \to X^{(\homdeg-1)}$ and every map homotopic to it has image intersecting the interior of $\tau$ non-trivially. When $(X, \partial X)$ is a polyhedral pair, this coincides with the usual notion of a facet.

Our description of the modules is an explicit version of the argument given in \cite[Example III.5.5b]{Brown1982}, but carried out over the rationals. Concretely, the $\QQ G$-modules $V_\sigma$ that we are about to describe are isomorphic to the modules $\mathrm{Ind}_{\group(\sigma)}^\group \Z_\sigma \otimes \QQ$ in Brown's notation. 

\iffalse
 which agrees with the $E^1$-page of the spectral sequence \eqref{eqn SS1} applied to a subspace of $X_N^\ast$ with trivial coefficients $\QQ$. Our choice of basis will be slightly different than usual, involving some rescaling.  Later this will allow us to upgrade $V_\bullet$ to a chain complex of free $\QQ\group$-modules. 
\fi

For each cell $\sigma$ of $X$, fix an orientation of $\sigma$.
If $\tau$ and $\sigma$ are cells of $X$ that are not contained in $\partial X$, we denote by $\Stabcoset(\tau, \sigma) \subset \group$ the (finite) set of all $g\in \group$ such that $\tau. g$ is a (not necessarily proper) face of $\sigma$, ignoring the orientations. Note that if $\Stabcoset(\tau, \sigma)$ is non-empty, then it is a double coset of the stabilisers $\Stabcoset(\tau) \coloneqq \Stab_{\group}(\tau)\leqslant \group$ on the left and $\Stabcoset(\sigma)$ on the right, where again the orientation is ignored.
Furthermore, if $g \in \Stabcoset(\tau,\sigma)$ and $\tau.g$ is either equal to $\sigma$ or a facet of $\sigma$, we define $\eta(\tau, \sigma, g)\in \{\pm 1\}$ to be $1$ if $g$ sends the fixed orientation of $\tau$ to the orientation on $\tau.g$ induced from that of $\sigma$, and to be $-1$ otherwise. Clearly, for $h \in \Stabcoset(\sigma)$ we have
\[
\eta(\tau,\sigma,g)\eta(\sigma,\sigma,h) = \eta(\tau,\sigma,gh) \ \textrm{ and } \ \eta(\sigma,\sigma,h) = \eta(\sigma,\sigma,h^{-1}),
\]
and similarly for $(h,g) \in \Stabcoset(\tau) \times \Stabcoset(\tau,\sigma)$ and $(h,g) \in \Stabcoset(\sigma) \times \Stabcoset(\tau,\sigma)$.

Let $\orbitrep_\homdeg$ be a set of representatives of the $\group$-orbits of unoriented $\homdeg$-cells of $X$ that are not contained in $\partial X$. For $\sigma\in \orbitrep_\homdeg$, we put
\begin{equation*}
	V_{\sigma} = \left\{ \sum_{g \in \Stabcoset(\sigma)} \eta(\sigma, \sigma, g) g \xi \,\middle|\, \xi \in \QQ \group \right\}.
\end{equation*}
One easily sees that this is isomorphic to Brown's $\mathrm{Ind}_{\group(\sigma)}^\group \Z_\sigma \otimes \QQ$  via the map taking a generator of $\Z_\sigma$ to $v_\sigma = \frac 1 {|\group(\sigma)|} \sum_{g \in \Stabcoset(\sigma)} \eta(\sigma, \sigma, g) g$. In other words, the element $v_\sigma$ represents the $\homdeg$-chain with weight one on $\sigma$, and zero elsewhere -- we will refer to this as the \emph{characteristic chain} of $\sigma$. 
In particular, for $h \in \Stabcoset(\sigma)$ we have $hv_\sigma = \eta(\sigma,\sigma,h)v_\sigma = v_\sigma h$.

We let $V_\homdeg=\bigoplus_{\sigma\in \orbitrep_\homdeg}V_{\sigma}$
and define the boundary operators $\partial_\homdeg \colon V_\homdeg\rightarrow V_{\homdeg-1}$ as follows.
Let $\sigma\in \orbitrep_\homdeg$ and $\tau \in \orbitrep_{\homdeg-1}$.
We will define the map $\partial_{\sigma,\tau} \colon V_\sigma \to V_\tau$, and then set 
\[\partial_\homdeg = \bigoplus_{\sigma \in \orbitrep_{\homdeg}} \sum_{\tau \in \orbitrep_{\homdeg-1}}\partial_{\sigma,\tau} \colon \bigoplus_{\sigma \in \orbitrep_{\homdeg}} V_\sigma \to \bigoplus_{\tau \in \orbitrep_{\homdeg-1}} V_\tau.\]

Define $\partial_{\sigma,\tau}\colon V_{\sigma} \to V_{\tau}$ to be the map given by left multiplication with  
\begin{equation*}
 \frac {1} {|\Stabcoset(\tau)|}   \sum_{g \in \Stabcoset(\tau,\sigma)} \eta(\tau, \sigma, g) g.
\end{equation*}
Here, $\eta(\tau, \sigma, g)$ is computed with respect to the fixed orientations on $\sigma$ and $\tau$. 
When $\tau.G$ does not contain faces of $\sigma$, we set $\partial_{\sigma,\tau}$ to be the zero map.

Let us compute $\partial_{\sigma, \tau} v_\sigma$ for $\sigma$ and $\tau$ as above. Picking orbit representatives, we write 
$\Stabcoset(\tau, \sigma) = \bigsqcup_{i=1}^l \Stabcoset(\tau) g_i$ for some collection of elements $g_1, \dots, g_l$ of $\Stabcoset(\tau,\sigma)$. In this notation, $\sigma$ has exactly $l$ facets in the orbit of $\tau$, namely $\tau.g_1, \ldots, \tau.g_l$. Now,
\begin{align*}
 \partial_{\sigma,\tau} v_\sigma =&  \frac {1} {|\Stabcoset(\tau)|}   \sum_{g \in \Stabcoset(\tau,\sigma)} \eta(\tau, \sigma, g) g v_\sigma \\
 =&  \frac {1} {|\Stabcoset(\tau)|}   \sum_{i=1}^l \sum_{g \in \Stabcoset(\tau)} \eta(\tau, \sigma, gg_i) gg_i v_\sigma \\
 =&  \frac {1} {|\Stabcoset(\tau)|}   \sum_{i=1}^l \sum_{g \in \Stabcoset(\tau)} \eta(\tau, \tau, g) \eta(\tau, \sigma, g_i) gg_i v_\sigma \\
  =&    \sum_{i=1}^l \eta(\tau, \sigma, g_i) v_\tau g_i v_\sigma \\
   =&    \frac {1} {|\Stabcoset(\sigma)|} \sum_{i=1}^l \sum_{h \in \Stabcoset(\sigma)} \eta(\tau, \sigma, g_i)  v_\tau g_i \eta(\sigma,\sigma,h)h \\
      =&    \frac {1} {|\Stabcoset(\sigma)|} \sum_{i=1}^l \sum_{h \in \Stabcoset(\sigma)} \eta(\tau, \sigma, g_ih)  v_\tau g_i h.
\end{align*}
For every $i$ and $h\in \Stabcoset(\sigma)$, we have $g_i h\in \Stabcoset(\tau,\sigma)$, so  there exists a unique $j$ such that $g_i h \in \Stabcoset(\tau) {g_j}$. Set $h'\coloneqq g_i h {g_j}^{-1} \in \Stabcoset(\tau)$. We have 
\begin{align*}
\eta(\tau, \sigma, g_ih) v_\tau g_i h &=  \eta(\tau, \sigma, h'g_j) v_\tau h' g_j \\ &= \eta(\tau, \sigma, h' g_j) \eta(\tau, \tau,h') v_\tau g_j \\ &= \eta(\tau, \sigma, g_j)  v_\tau g_j.
\end{align*}
Moreover, the map $(i,h) \mapsto j$ is $|\Stabcoset(\tau)|$-to-one, and hence
\[
\partial_{\sigma,\tau} v_\sigma = \frac {1} {|\Stabcoset(\sigma)|} \sum_{i=1}^l \sum_{h \in \Stabcoset(\sigma)} \eta(\tau, \sigma, g_ih)  v_\tau g_i h = v_\tau \cdot  \sum_{j=1}^l  \eta(\tau, \sigma, g_j)   g_j.
\]

Since $v_\tau$ represents the characteristic chain of the cell $\tau$, the expression above is precisely the sum of the characteristic chains of the $\group$-translates of $\tau$ that are faces of $\sigma$, with signs  depending on orientations.
Therefore the chain complex $V_\bullet = (V_\homdeg, \partial_\homdeg)$ is isomorphic to the chain complex of the CW-pair $(X, \partial X)$.

Let $M$ be any $\QQ G$-module, and 
consider $H_0(\group; V_p \otimes M)$; by definition, this is the abelian group of $\QQ G$-coinvariants of $V_p \otimes M$, but this in turn is easily seen to be precisely $V_p \otimes_{\QQ \group} M$.
Clearly, the differentials in $V_\bullet \otimes M$ descend to those of $V_\bullet \otimes_{\QQ G} M$, and therefore  $V_\bullet \otimes_{\QQ G} M$ coincides with the chain complex with terms $H_0(\group; V_p \otimes M)$
 that appears as the zeroth row (that is, $q=0$) of the first page of the spectral sequence computing $H^\group_{p+q} (X, \partial X; M)$, see \cite[Equation VII.7.6]{Brown1982}. Crucially, the other terms are all zero: by \cite[Equation VII.7.6]{Brown1982} again, they are all equal to direct sums of homologies in degree $q$ of the groups $\group(\sigma)$ for various cells $\sigma$. These groups are all finite, and the rational cohomological dimension of a finite group is zero. Hence, for $q \neq 0$, these terms are all zero, as claimed. We therefore see that $H^\group_{p} (X, \partial X; M)$ coincides with  $H_{p}(V_\bullet \otimes_{\QQ G} M)$ for every $\QQ G$-module $M$. 

The modules $V_n$ constructed above are submodules of free modules. We will now show that they are direct summands.
To this end, we will first decompose $\QQ \group$.
 \begin{lemma}\label{lemma:qg_decomposition}
 	Let $K$ be a finite subgroup of $\group$ and $\eta \colon  K \rightarrow\{\pm 1\}$ be a homomorphism. Put
 	\begin{gather*}
 		V = \left\{  \sum_{k \in K} \eta(k) k \xi \,\middle|\, \xi \in \QQ \group \right\}
 		\\\text{and}\quad
 		W=\left\{\sum_{g\in \group}\lambda(g)g\in\QQ \group\,\middle|\,\sum_{k\in K}\eta(k)\lambda(kg)=0\quad\forall g\in \group\right\}.
 	\end{gather*}
 	Then $\QQ \group=V\oplus W$.
 \end{lemma}
 \begin{proof}
 	Consider the $\QQ G$-linear map $\rho \colon \QQ G \to \QQ G$ given by $x \mapsto vx$ where $v=\frac{1}{|K|}\sum_{k\in K}\eta(k)k$. Clearly, the image of $\rho$ is precisely $V$, and since $v^2 = v$, the map $\rho$ is actually a retraction of $\QQ G$ onto $V$. It follows that $\QQ G = \ker \rho \oplus V$, and hence it is enough to show that $W = \ker \rho$, but this is immediate.
 	\iffalse
 	We prove first that $\QQ \group$ is the algebraic sum $V+W$. It suffices to show that each $g\in \group$ decomposes in $\QQ \group$ into a sum $g=v_g+w_g$, where $v_g\in V$ and $w_g\in W$. Indeed, such a decomposition can be obtained by taking $v_g=\frac{1}{|K|}\sum_k\eta(k)kg$ and $w_g=g-v_g$.
 	
 	To obtain a direct sum decomposition of $\QQ \group$, we still have to prove that the intersection $V\cap W$ is trivial. Take $u\in V\cap W$. Since $u$ belongs to $V$, we can express it as follows:
 	\[
 	u=\frac{1}{|K|}\sum_{k\in K}\eta(k)k\sum_{g\in \group}\lambda(g)g=\frac{1}{|K|}\sum_{g\in \group}\underbrace{\left(\sum_{k\in K}\eta(k)\lambda(k^{-1}g)\right)}_{\beta(g)}g.
 	\]
 	As u is in $W$, we have $\sum_k\eta(k)\beta(kg)=0$ for all $g\in G$.
 	That is, 
 	\begin{equation} \label{eqnarray:coeff_ind}
 		\sum_{k\in K} f_g(k)=0, \text{ where }f_g(k)=\sum_{l\in K}\eta(kl)\lambda(l^{-1}kg). 
 	\end{equation}
 	We note that the value $f_g(k)$ does not depend on $k\in K$ -- for every $k_1,k_2\in K$, we have
 	\begin{align*}
 		f_g(k_2)&=\sum_l\eta(k_2l)\lambda(l^{-1}k_2g)=\sum_l\eta(k_2^2k_1^{-1}l)\lambda(l^{-1}k_1g)\\
 		&=\sum_l\eta(k_1l)\lambda(l^{-1}k_1g)=f_g(k_1).
 	\end{align*}
 	Coming back to \ref{eqnarray:coeff_ind},
 	\begin{align*}
 		\sum_k f_g(k)=0&\Leftrightarrow f_g(1)=0\\
 		&\Leftrightarrow\sum_l\eta(l)\lambda(l^{-1}g)=0\Leftrightarrow\beta(g)=0,
 	\end{align*}
 	and therefore $u=0$.\fi
 \end{proof}
 
 It follows
  that \begin{equation}
 	\label{projective}
 	\left(\QQ \group\right)^{\orbitrep_\homdeg} = \bigoplus_{\orbitrep_{\homdeg}} \QQ \group =V_n\oplus W_n
 \end{equation} for
 \begin{gather*}
 	V_n=\bigoplus_{\sigma \in \orbitrep_\homdeg}V_{\sigma}
 	\quad\text{and}\quad
 	W_n=\bigoplus_{\sigma \in \orbitrep_\homdeg}W_{\sigma},
 \end{gather*}
 where
 \[
 W_{\sigma}=\left\{\sum_{g\in \group}\lambda(g)g\in\QQ \group\,\middle|\,\sum_{k \in \Stabcoset(\sigma)} \eta(\sigma, \sigma, k) \lambda(kg)=0\quad\forall g\in \group\right\}.
 \]
 In particular, this implies that $V_\bullet$ is a chain complex of projective $\QQ \group$-modules.

\section{Homology with unitary coefficients}\label{subsection:general_stuff}
\label{sec_homology_unitary_coeffs}
We now compute homology with coefficients being a Hilbert space endowed with a unitary action of $\group$. In this context, one usually considers \emph{reduced} homology, namely kernels of the differentials divided by the closures of the images. The reason is that the resulting abelian groups are then Hilbert spaces themselves. This has practical consequences, for example when one wants to access the von Neumann dimension, like in the theory of $L^2$-homology.
The usual homology is known as the \emph{non-reduced} homology in this context.

Let $\mathcal H$ be a finite-dimensional Hilbert space (either real or complex)  endowed with a linear $\group$-action, where $\group = \SL_{\rankSL}(\Z)$. 
In practice, this is going to be one of our carefully chosen orthogonal  representations. To compute $H^{\group}_{\bullet}(X_{\rankSL}^\ast,\partial X_{\rankSL}^\ast;\mathcal H)$ we need to tensor the chain complex $V_\bullet$ defined in \cref{sec_chain_complex_Voronoi} with $\mathcal H$ over $\mathbb Q G$  and compute the homology of the resulting chain complex
\begin{eqnarray*}
	\ldots\rightarrow V_{\homdeg+1}\otimes_{\QQ G} \mathcal{H}\xrightarrow{\partial_{\homdeg+1}\otimes\id}V_\homdeg\otimes_{\QQ G} \mathcal{H}\xrightarrow{\partial_{\homdeg}\otimes\id}V_{\homdeg-1}\otimes_{\QQ G} \mathcal{H}\rightarrow\ldots .
\end{eqnarray*}

Using the isomorphism 
$\QQ \group\otimes_{\QQ \group}\mathcal{H}\xrightarrow{\cong}\mathcal{H}$, $g\otimes v\mapsto\pi(g)v$ and \cref{projective}, we view the modules  $V_\homdeg \otimes_{\QQ \group} \mathcal H$ as direct summands of Hilbert spaces $  \QQ \group^{\orbitrep_{\homdeg}} \otimes_{\QQ \group} \mathcal H \cong \mathcal H^{\orbitrep_{\homdeg}}$; hence, the modules $V_\homdeg \otimes_{\QQ \group} \mathcal H$ are Hilbert spaces themselves.

Since $\mathcal  H$ is finite dimensional, and the $\QQ G$-modules $V_i$ are direct summands of finitely generated free $\QQ \group$-modules by \cref{projective}, all the linear spaces  appearing as images of the differentials  in the  chain complex are finite dimensional. Hence they are closed, which implies that the reduced and non-reduced homologies of $V_\bullet \otimes_{\QQ \group} \mathcal H$ coincide. This is important: For us it is easier to compute reduced homology as the kernel of a Laplacian, but to use the previous section we need to determine the non-reduced homology.

It follows from the Hodge decomposition  that $H^{\group}_{\homdeg}(X_{\rankSL}^\ast,\partial X_{\rankSL}^\ast;\mathcal H) = \ker \Delta_\homdeg$ where
\[
\Delta_{\homdeg}=\left(\partial_\homdeg^*\partial_\homdeg + \partial_{\homdeg+1}\partial_{\homdeg+1}^*\right)\otimes\id \colon V_\homdeg\otimes_{\QQ \group}\mathcal{H}\rightarrow V_\homdeg\otimes_{\QQ \group}\mathcal{H}
\]
is the Laplacian. This version of the Hodge decomposition has exactly the same proof as the usual $L^2$-version, see for example \cite[Lemma 2.0.2]{Schreve2015}.
\iffalse
\begin{lemma}\label{lemma:hodge_iso}
	Let $\partial_{\homdeg+1}:C_{\homdeg+1}\rightarrow C_{\homdeg}$ and $\partial_{\homdeg}:C_{\homdeg}\rightarrow C_{\homdeg-1}$ be two maps between Hilbert spaces such that $\partial_\homdeg\circ\partial_{\homdeg+1}=0$ and denote by $\Delta=\partial_\homdeg^*\partial_\homdeg+\partial_{\homdeg+1}\partial_{\homdeg+1}^*:C_\homdeg\rightarrow C_\homdeg$ the Laplacian. Then the map
	\begin{gather*}
		\varphi:\Ker\Delta\rightarrow\Ker\partial_\homdeg/\overline{\mathrm{Im}\partial_{\homdeg+1}},
		\quad
		u\mapsto[u]
	\end{gather*}
is an isomorphism of abelian groups.
\end{lemma}
\begin{proof}
	Note that $\Ker\Delta=\Ker\partial_\homdeg\cap\Ker\partial_{\homdeg+1}^*$. Suppose that $\varphi(u)=0$ for some $u\in\Ker\Delta$. This means that $u\in\overline{\mathrm{Im}\partial_{\homdeg+1}}$. On the other hand, $\partial_{\homdeg+1}^*u=0$ which, in connection with $\left(\Ker\partial_{\homdeg+1}^*\right)^{\perp}=\overline{\mathrm{Im}\partial_{\homdeg+1}}$, yields $u=0$. Since $u$ was arbitrary, this means that $\varphi$ is injective.
	
	For the proof of surjectivity, take an arbitrary $u\in\Ker\partial_\homdeg$ and decompose it as $u=v+v^{\perp}$, where $v\in\overline{\mathrm{Im}\partial_{\homdeg+1}}$ and $v^{\perp}\in\left(\overline{\mathrm{Im}\partial_{\homdeg+1}}\right)^{\perp}\cap\Ker\partial_\homdeg=\Ker\partial_{\homdeg+1}^*\cap\Ker\partial_\homdeg=\Ker\Delta$. The equality $[u]=[v^{\perp}]$ concludes then the proof.
\end{proof}
\fi

Instead of computing the kernel of the Laplacian $\Delta_{\homdeg}$, we substitute it with another operator $\Delta_{\homdeg}'\colon \left(\QQ \group\right)^{\orbitrep_\homdeg}\otimes_{\QQ G} \mathcal{H}\rightarrow \left(\QQ \group\right)^{\orbitrep_\homdeg}\otimes_{\QQ G} \mathcal{H}$.
 This is done in such a way that $\ker\Delta_{\homdeg}\cong\ker\Delta_{\homdeg}'$. We then compute $\ker\Delta_{\homdeg}'$.

We define the operator $\Delta_\homdeg'$ to be equal to $\Delta_\homdeg$ on the $V_\homdeg$ component and to be the identity on the $W_\homdeg$ component of \cref{projective}:
\[
\Delta_\homdeg'\coloneqq\left(\left(\partial_\homdeg^*\partial_\homdeg+\partial_{\homdeg+1}\partial_{\homdeg+1}^*\right)\oplus\id_{W_n}\right)\otimes\id_{\mathcal{H}}:(V_n\oplus W_n)\otimes_{\QQ \group}\mathcal{H}\rightarrow(V_n\oplus W_n)\otimes_{\QQ \group}\mathcal{H}.
\]

Since tensoring preserves direct sums, we have
\[
\ker \Delta'_\homdeg = \ker \Delta_\homdeg \oplus \ker (\id \otimes \id \colon W_n \otimes \mathcal{H} \to W_n \otimes \mathcal{H}) \cong \ker \Delta_\homdeg. 
\]

The reason for working with $\Delta_\homdeg'$ instead of $\Delta_\homdeg$ is that we can easily describe $\Delta'_\homdeg$ as a matrix by evaluating the representation $\pi \colon \group\rightarrow\mathcal{B}(\mathcal{H})$ on 
 \[
\left(\partial_\homdeg^*\partial_\homdeg+\partial_{\homdeg+1}\partial_{\homdeg+1}^*\right)\oplus\id_W \colon \left(\QQ \group\right)^{\orbitrep_\homdeg}\rightarrow\left(\QQ \group\right)^{\orbitrep_\homdeg}.
\]
 The key point here is that $\Delta'_\homdeg$ is obtained from a homomorphism of free $\QQ \group$-modules.
The Laplacian $\Delta_\homdeg$ on the other hand is obtained from a map of projective $\QQ G$-modules.

\section{Nontriviality of cohomology}\label{section:cohomology_nontriviality}
In this section we describe finite-dimensional orthogonal representations $\pi_\rankSL \colon \SL_\rankSL(\ZZ)\rightarrow\mathcal{B}(\mathcal{H}_\rankSL)$, $\rankSL=3,4$, all of whose invariant vectors are trivial,  such that the cohomology $H^{\rankSL-1}(\SL_\rankSL(\ZZ),\pi_\rankSL)$ is non-zero. Since the representations are finite dimensional, the cohomology coincides with the reduced cohomology.

The general scheme is as follows. We find, for some prime $p_\rankSL$, an orthogonal (hence unitary) representation $\pi_\rankSL'\colon \SL_\rankSL(\ZZ_{p_\rankSL})\rightarrow\mathcal{B}(\mathcal{H}_\rankSL)$  all of whose invariant vectors are trivial, where $\ZZ_{p_\rankSL} = \ZZ/ {p_\rankSL}\ZZ$. This defines 
\[
\pi_\rankSL\colon \SL_\rankSL(\ZZ)\rightarrow\mathcal{B}(\mathcal{H}_\rankSL)
\]
by precomposing $\pi_\rankSL'$ with the modular map $\SL_\rankSL(\ZZ)\twoheadrightarrow \SL_\rankSL(\ZZ_{p_\rankSL})$. Applying $\pi_\rankSL$ to the operator $\Delta'_\homdeg$ yields a real square matrix, as explained at the end of \cref{sec_homology_unitary_coeffs}. To show that $H^{\rankSL-1}(\SL_\rankSL(\ZZ),\pi_\rankSL)\neq 0$, we have to prove that for $n=N(N-1)/2$, this matrix has non-trivial kernel (see \cref{eqn:SpecSeq}). Moreover, computing the real dimension of this kernel gives precisely the real dimension of the corresponding cohomology. 

Let us define $\pi_3'$ and $\pi_4'$. In both cases, we indicate a subgroup $H_\rankSL$ of $\SL_\rankSL(\ZZ_{p_\rankSL})$, $\rankSL=3,4$, and an orthogonal representation $\pi_\rankSL''$ of $H_\rankSL$  all of whose invariant vectors are trivial. The representation $\pi_\rankSL'$ is the representation induced from $\pi_\rankSL''$, from $H_\rankSL$ to $\SL_\rankSL(\ZZ_{p_\rankSL})$, where $p_3=3$, $p_4=2$. In order to get $\pi_N''$, in turn, we proceed as follows. We indicate a normal subgroup $H$ of $H_N$ and define an orthogonal  representation $\rho$ of the quotient group $H_N/H$  all of whose invariant vectors are trivial, in an explicit way. As in the case of constructing $\pi_N$ from $\pi_N'$, the representation $\pi_N''$ is defined by precomposing $\rho$ with the quotient homomorphism $H_N\twoheadrightarrow H_N/H$. Below, we describe the representation $\rho$ for $N=3$ and $N=4$.
\begin{enumerate}[topsep=5pt, partopsep=5pt, itemsep=0pt, parsep=5pt]
	\item \emph{The case $\rankSL=3$}. We set $H_3$ to be the subgroup of $\SL_3(\ZZ_3)$ generated by the two matrices
	\begin{align*}
		s=\begin{pmatrix}
			0&0&1\\
			0&2&0\\
			1&1&0
		\end{pmatrix}\quad\text{and}\quad t=\begin{pmatrix}
			1&2&0\\
			0&2&0\\
			1&1&2
		\end{pmatrix}.
	\end{align*}
The group 	$H_3$ is isomorphic to $S_3\times S_3$ and we take its index-two subgroup $H\cong C_3\times S_3$ generated by $s$ and the two matrices
	\begin{align*}
		a=\begin{pmatrix}
			0&0&1\\
			0&2&0\\
			1&1&0
		\end{pmatrix}\quad\text{and}\quad b=\begin{pmatrix}
			0&1&2\\
			0&1&0\\
			1&2&2
		\end{pmatrix}.
	\end{align*}
	This allows us to define $\rho:H_N/H\rightarrow\GL_{1}(\RR)$ as follows.
	\begin{align*}
		\rho(hH)=\begin{cases*}
			(1) & if $h\in H$ \\
			(-1)        & if $h\notin H$
		\end{cases*}.
	\end{align*} 
	Suppose that $v\in\RR$ is an invariant vector of $\rho$. Then any $h\in H_N\setminus H$ represents the generator $hH$ of $H_N/H$. Thus, $\rho(hH)v=-v$. Since $v$ is invariant this means that it is the zero vector.
	
	\item \emph{The case $\rankSL=4$}. The group $H_4$ is the subgroup of $\SL_4(\ZZ_2)$ generated by the two matrices
	\begin{align*}
		s=\begin{pmatrix}
			1&0&0&0\\
			0&0&0&1\\
			1&1&0&1\\
			1&0&1&1
		\end{pmatrix}\quad\text{and}\quad t=\begin{pmatrix}
			0&1&1&0\\
			0&1&1&1\\
			1&1&1&1\\
			0&0&1&1
		\end{pmatrix}.
	\end{align*}
	The order of $H_4$ is $576$ and $H_4$ possesses a normal subgroup $H$ generated by the following six matrices:
	\begin{align*}
		a&=\begin{pmatrix}
			1&0&1&1\\
			0&1&1&1\\
			0&0&1&0\\
			0&0&0&1
		\end{pmatrix}\text{,}\quad b=\begin{pmatrix}
			1&1&1&0\\
			1&0&0&0\\
			0&0&1&0\\
			1&0&1&1
		\end{pmatrix}\text{,}\quad c=\begin{pmatrix}
			0&1&1&1\\
			1&0&1&1\\
			0&0&1&0\\
			0&0&0&1
		\end{pmatrix}\\d&=\begin{pmatrix}
			1&0&1&0\\
			0&1&1&0\\
			0&0&1&0\\
			0&0&0&1
		\end{pmatrix}\text{,}\quad e=\begin{pmatrix}
			0&1&0&1\\
			0&1&0&0\\
			0&0&1&0\\
			1&1&0&0
		\end{pmatrix}\text{,}\quad f=\begin{pmatrix}
			1&0&0&0\\
			1&0&1&1\\
			0&0&1&0\\
			1&1&1&0
		\end{pmatrix}.
	\end{align*}
	The quotient $H_4/H$ is isomorphic to $D_6\cong S_3$, the dihedral group of order six, and is generated by the equivalence classes of 
	\begin{align*}
		x=\begin{pmatrix}
			0&1&0&0\\
			0&1&1&1\\
			1&1&1&1\\
			0&0&0&1
		\end{pmatrix}\quad\text{and}\quad y=\begin{pmatrix}
			0&1&0&0\\
			0&0&0&1\\
			1&1&0&1\\
			0&1&1&1
		\end{pmatrix}.
	\end{align*}
	Let us denote by $M(\sigma)$ the permutation matrix of $\sigma \in S_3$, e.g.
	\begin{align*}
		M(\cycle{1,2,3})=\begin{pmatrix}
			0&0&1\\
			1&0&0\\
			0&1&0
		\end{pmatrix}.
	\end{align*}
	We define $\rho:H_N/H\rightarrow\GL_{3}(\RR)$ as follows.
	\begin{align*}
		\rho(hH)=\begin{cases*}
			I_3 & if $hH=H$, \\
			M(\cycle{1,2,3})       & if $hH=xH$,\\
			M(\cycle{3,2,1})       & if $hH=x^2H$,\\
			-M(\cycle{1,2})       & if $hH=yH$,\\
			-M(\cycle{2,3})       & if $hH=yxH$,\\
			-M(\cycle{1,3})       & if $hH=yx^2H$.
		\end{cases*}
	\end{align*} 
Assuming $v=(v_1,v_2,v_3)\in\RR^3$ is an invariant vector of $\rho$, we have $(v_1,v_2,v_3)=\rho(xH)v=(v_3,v_1,v_2)$. This already implies that $v_1=v_2=v_3=t$. On the other hand, $(t,t,t)=\rho(yH)=(-t,-t,t)$ which means $t=0$. Thus, the only invariant vector of $\rho$ is the zero vector.
\end{enumerate}
Denote by $\Xi_3$ the operator $\Delta_3'$ for $\rankSL=3$, as constructed at the end of \cref{subsection:general_stuff}. Similarly, denote by $\Xi_4$ the operator $\Delta'_6$ for $\rankSL=4$.

After performing the necessary computations, we were able to show the main result of this paper.
\begin{thm}\label{theorem:nontrivial_cohom}
	The coranks of the matrices $\pi_3(\Xi_3)$ and $\pi_4(\Xi_4)$ are equal to $4$ and $2$ respectively. Therefore, $H^2(\SL_3(\ZZ),\pi_3)\cong\RR^4$ and $H^3(\SL_4(\ZZ),\pi_4)\cong\RR^2$.
\end{thm}

\section{Implementation}\label{section:implementation}
In order to get our results, we implemented the necessary procedures in Julia \cite{bezanson2017julia}. They are available at \cite{HigherTSLN2024}. Below we describe them in more detail.
\subsection{Voronoi tessalation and the chain complex $V_\bullet$}

In the following, we describe how we computed the equivariant chain complex $V_\bullet$ of $(X_\rankSL^*, \partial X_\rankSL^*)$ described in \cref{sec_chain_complex_Voronoi}. We again set $G = \SL_\rankSL(\ZZ)$.

\subsubsection{Barycentres}
\label{sec_barycentres}
For implementing computations around the Voronoi cell structure on $X_\rankSL^*$, we use that much information about the cells of this complex can be inferred from knowing their barycentres:
For a cell $\sigma$, let $q(\sigma)\coloneqq \sum_{v\in m(\sigma)} \hat{v}$. Then $\pi(q(\sigma))$ is the barycentre of $\sigma$.
If $\sigma$ and $\sigma'$ are cells of the same dimension, then $g\in G$ sends $\sigma$ to $\sigma'$ if and only if it sends $q(\sigma)$ to $q(\sigma')$, i.e.
\begin{equation*}
	\Stabcoset(\sigma, \sigma') = \ls g \in G \ms \sigma.g = \sigma' \rs = \ls g \in G \ms q(\sigma).g = q(\sigma') \rs.
\end{equation*}
In particular, the setwise stabiliser of a cell $\sigma$ is given by
\begin{equation}
\label{eq_stab_sigma}
	\Stab_{G}(\sigma) = \ls g \in G \ \ms q(\sigma).g = q(\sigma) \rs.
\end{equation}
Furthermore, a cell $\sigma$ intersects $X_\rankSL^*$ non-trivially (so it is not contained in $\partial X_\rankSL^*$) if and only if its barycentre does, which is equivalent to saying that $q(\sigma)$ is positive definite.
Positive definiteness can be effectively checked with a computer, so this allows us to determine which cells lie in $\partial X_\rankSL^*$. And if $q(\sigma)$ and $q(\sigma')$ are positive definite forms, then the set of all $g\in G$ sending $q(\sigma)$ to $q(\sigma')$ can be effectively computed as well. This allows one to determine $\Stabcoset(\sigma, \sigma')$ and $\Stab_{G}(\sigma) = \Stabcoset(\sigma)$.

\subsubsection{Orientations}
\label{orientations}
An orientation of a Voronoi cell $\sigma$ is the same as an orientation of the $\dim(\sigma)$-dimensional vector space $\RR(\sigma)$ of symmetric $\rankSL\times \rankSL$ matrices spanned by the forms $\hat{v}$ with $v\in m(\sigma)$. 
In practice, we determine such an orientation by picking an ordered basis of $\RR(\sigma)$.
If $\tau$ is a facet of $\sigma$ and both have a fixed orientation, we compute the relative orientation $\eta(\tau,\sigma,1)$ as follows: Start with the ordered basis $B$ of $\RR(\tau)$; let $B'$ be the basis of $\RR(\sigma)\supset \RR(\tau)$ obtained by appending to $B$ any vector $\hat{v}$ with $v\in m(\sigma) \setminus m(\tau)$ (the result does not depend on the choice of $\hat{v}$). Then $\eta(\tau, \sigma, 1) = \pm 1$ is the orientation of $B'$ in the oriented vector space $\RR(\sigma)$.

\subsubsection{Computing $V_\bullet$ Step 1: Voronoi cells}
As a first step to compute $V_\bullet$, we compute for each $\homdeg$ a set $\orbitrep_\homdeg$ of representatives of the $\group$-orbits of $\homdeg$-cells in $X_{\rankSL}^\ast$ that intersect $X_{\rankSL}$ non-trivially.
Such cells occur in each dimension $\rankSL-1\leq \homdeg \leq \dim(X_\rankSL) = \rankSL(\rankSL + 1)/2 -1$.
Each representative $\sigma\in \orbitrep_\homdeg$ is saved in terms of $m(\sigma)$, a finite set of vectors in $\ZZ^\rankSL$. We ignore orientations for now.

We start with the top-dimensional cells, $\homdeg = \dim(X_\rankSL)$. These are in 1-to-1 correspondence with the $\group$-orbits of rank-$\rankSL$ perfect forms. 
These orbits have been computed up to $\rankSL = 8$ \cite{DutourSikiric2007}. 
We extract the information from the \texttt{Lattices} database of Nebe and Sloane \cite{NebeSloane}.
(Strictly speaking, these are the $\GL_\rankSL(\ZZ)$-orbits of perfect forms, but for $\rankSL \leqslant 5$, these are the same as the $\SL_\rankSL(\ZZ)$-orbits.)
For each perfect form $q$, we compute its minimal vectors $m(q)$ using GAP \cite{GAP4}.

Now assume that we have computed $\orbitrep_{\homdeg+1}$. We then compute $\orbitrep_{\homdeg}$ as follows:
For every $\sigma\in \orbitrep_{\homdeg+1}$, the convex hull of the forms $\hat{v}$ with $v\in m(\sigma)$ is a polyhedral subset of $K_\rankSL^*$. We compute each $\hat{v} = v v^t $ as a symmetric matrix in $\ZZ^{\rankSL\times \rankSL}$ such that their convex hull is a subset of $\RR^{\rankSL\times \rankSL}$.
We use the \texttt{Polyhedra} package \cite{legat2023polyhedral,Polyhedra} with the exact version of the library \texttt{CDDLib} \cite{CDDLib} in its Julia wrapper \texttt{CDDLib.jl} \cite{CDDLib.jl} to compute all facets of this subset. 
For each such facet, we first determine whether the corresponding cell $\tau$ of $X_\rankSL^*$ intersects $X_\rankSL$ non-trivially. As mentioned in \cref{sec_barycentres}, this is equivalent to $q(\tau)$ being positive definite. We check this using Sylverster's criterion, which allows for exact computations (see also \cref{sec_rigour}).
If  $q(\tau)$ is positive definite, we check whether we already added a representative of the $\group$-orbit of $\tau$ to $\orbitrep_{\homdeg}$ in a previous step. To do so, we check whether $q(\tau)$ lies in the orbit of $q(\tau')$ for some $\tau'\in \orbitrep_{\homdeg}$.
This is done using an algorithm of Plesken--Souvignier \cite{Plesken1997}. We used an implementation of this algorithm by Brandt \cite{Brandt} using a combination of Julia and GAP, which we adapted to our purposes.

\subsubsection{Computing $V_\bullet$ Step 2: Differentials}
We next determine for every $\homdeg$ the modules $V_n$ and the differential $\partial_\homdeg \colon V_\homdeg\rightarrow V_{\homdeg-1}$.
To do so, we first fix for all $\homdeg$ and all $\sigma\in \orbitrep_\homdeg$ an orientation by computing an (arbitrary) ordered basis of the vector space $\RR(\sigma)$, as described in \cref{orientations}.

For each $\sigma\in \orbitrep_\homdeg$, we compute its stabiliser $\Stabcoset(\sigma)$ using \cref{eq_stab_sigma}. For every $g\in \Stabcoset(\sigma)$, we also compute $\eta(\sigma,\sigma,g)$ by comparing the fixed orientation of $\sigma$ with the image of this orientation under $g$. This determines the element $v_\sigma$, and hence the summand $V_\sigma\leqslant  V_\homdeg$. The module $V_n$ is the direct sum of modules $V_\sigma$. 

The differentials
 $\partial_\homdeg$ are defined as sums of the $\QQ \group$-morphisms $\partial_{\sigma,\tau}$, for $\sigma\in \orbitrep_\homdeg$ and $\tau\in \orbitrep_{\homdeg-1}$. 
To determine these, we first compute all facets of $\sigma$ using \texttt{Polyhedra.jl}. For each such facet $\tau'$, we determine whether $\tau'$ intersects the interior of $X_\rankSL^*$ by checking whether $q(\tau')$ is positive definite. If this is not the case, we ignore $\tau'$ and continue with the next facet of $\sigma$. If $q(\tau')$ is positive definite, we fix an orientation on $\tau'$.
We use the algorithm by Plesken--Souvignier \cite{Plesken1997,Brandt} to determine the unique $\tau\in \orbitrep_{n-1}$ that lies in the same $G$-orbit as $\tau'$. The algorithm also allows us to get a list of all $g\in \group$ such that $\tau.g = \tau'$. Comparing the orientation on $\tau'$ with the $g$-image of the fixed orientation on $\tau$ gives us $\eta(\tau,\tau',g)$ for all such $g$. We compute $\eta(\tau',\sigma,1)$ by comparing the orientation on $\tau'$ with the fixed orientation on $\sigma$, as explained in \cref{orientations}. We then compute $\eta(\tau,\sigma,g) = \eta(\tau,\tau.g,g) \cdot \eta(\tau.g,\sigma,1)$.
This is all the information that is necessary to determine 
\[
 x_{\tau'}\coloneqq \frac {1} {|\Stabcoset(\tau)|}   \sum_{g \in \Stabcoset(\tau,\tau')} \eta(\tau, \sigma, g) g.
 \]
We obtain $\partial_{\sigma,\tau}\colon V_\sigma\to V_\tau$ as left multiplication with the sum of all $x_{\tau'}$ where $\tau'$ is a facet of $\sigma$.
This sum is an element of the group ring $\QQ \group$ that we store to represent $\partial_{\sigma,\tau}$; we store $\partial_\homdeg$ as a matrix over $\QQ \group$.
We use the group ring implementation from \cite{StarAlgebras} (as used in \cite{KNO2019,KKN2021}) wrapped in the matrix setting in \cite{LowCohomologySOS}.

\subsection{Computing the Laplacians}
In order to compute the operators $\Xi_3$ and $\Xi_4$ as matrices over group rings, we start by computing the Laplacians $\Delta_\homdeg$. 
To work with a group ring, one has to be able to solve the word problem in the group. It turns out that an efficient way of computing with $\RR \group$ is to pre-compute a big enough portion of the Cayley graph of $\group$.

We fix $(\rankSL,\homdeg)$ to be $(3,3)$ or $(4,6)$, see \cref{section:cohomology_nontriviality}. 
We create a subspace $\mathbb{R}(E^{-1}E)$ of $\mathbb{R}\group$ supported on the set $E^{-1}E$, where $E$ consists of the elements of $\group$ appearing  in the support of $\partial_{\sigma,\tau}$ for any cells $\tau$ and $\sigma$ of dimension $\homdeg-1, \homdeg$, or $\homdeg+1$, and of elements of the stabiliser of any cell. The computer verifies that this set is actually big enough, that is, that $E^{-1}E$ contains the supports of all the group elements that appear in our computations. This was not clear a priori.
 
More precisely, $\mathbb{R}(E^{-1}E)$ is the subspace of $\mathbb{R}G$ consisting of the sums $\sum_{g\in E^{-1}E}\lambda_gg$ with twisted  multiplication $(x,y) \mapsto x^*y$ defined on $E$ only. In that way we ensure that the twisted multiplication is an intrinsic operation in $\mathbb{R}(E^{-1}E)$, defined on a subset of $\mathbb{R}(E^{-1}E)$ consisting of group ring elements supported on $E$. We now compute $\Delta_\homdeg$ as a matrix over $\RR(E^{-1}E)$. 

In the next step, we pass from $\Delta_\homdeg$ to $\Delta'_\homdeg$. We store $\Delta_\homdeg$ as an $\orbitrep_{\homdeg}\times\orbitrep_{\homdeg}$ matrix over $\QQ \group$.
We first need to make sure that the matrix only operates on $V_\homdeg$, rather than on $\bigoplus_{\orbitrep_{\homdeg}} \QQ \group$. To arrange this, one can
 multiply the matrix on both sides by a diagonal matrix with entries $v_\sigma$ in the $(\sigma,\sigma)$ position. However, it follows from the computations in \cref{sec_chain_complex_Voronoi} that this multiplication does not change the matrix (this is also verified in the code).
 Finally,  we add the diagonal matrix with entries $1-v_\sigma$ in the $(\sigma,\sigma)$ position to our matrix, to guarantee that we operate as the identity on $W_\homdeg$. This yields the desired matrix $\Delta'_\homdeg$, which is precisely $\Xi_3$ or $\Xi_4$, depending on $(\rankSL,\homdeg)$.
	
These computations constitute the \texttt{sln\_laplacians.jl} script from \cite{HigherTSLN2024}.

\subsection{Proving non-triviality of cohomology}
At this stage, we load the matrix $\Xi_3$ or $\Xi_4$ from the previous step, compute the orthogonal representation of $\SL_N(\ZZ)$ for $N=3$ and $N=4$ as described in \cref{section:cohomology_nontriviality} and apply this representation to the matrix. We compute the corank (or nullity) of the matrices, proving \cref{theorem:nontrivial_cohom}. All this is performed in the  \texttt{sln\_nontrivial\_cohomology.jl} script.
\subsubsection{Computing the representation $\pi_N''$}
In the first step, we compute the representation $\pi_\rankSL''$ with the \texttt{flip\_permutation\_representation} function. This boils down to computing the representation $\rho$ of the quotient $H_N/H$ (cf. \cref{section:cohomology_nontriviality}). To get the subgroups $H_N$, and $H$, we used GAP \cite{GAP4}. We decided, however, to define the generating matrices of these subgroups (obtained via GAP) entirely in Julia, due to simplicity of implementation. In order to check that the aforementioned subgroups are as described, one can use GAP. First, define the generators and the groups $H_N$ and $H$ as follows.
\begin{itemize}[leftmargin=*]
	\item \emph{The case} $N=3$.
	\texttt{\\0\_:=0*Z(3)\^{}0; 1\_:=1*Z(3)\^{}0; 2\_:=2*Z(3)\^{}0;\\
		s:=[[0\_,0\_,1\_],[0\_,2\_,0\_],[1\_,1\_,0\_]];\\
		t:=[[1\_,2\_,0\_],[0\_,2\_,0\_],[1\_,1\_,2\_]];\\
		a:=[[0\_,0\_,1\_],[0\_,2\_,0\_],[1\_,1\_,0\_]];\\
		b:=[[0\_,1\_,2\_],[0\_,1\_,0\_],[1\_,2\_,2\_]];\\	
		H\_N:=Group([s,t]);\\
		H:=Group([s,a,b]);
	}
	\item \emph{The case} $N=4$.
	\texttt{\\0\_:=0*Z(2)\^{}0; 1\_:=1*Z(2)\^{}0;\\
		s:=[[1\_,0\_,0\_,0\_],[0\_,0\_,0\_,1\_],[1\_,1\_,0\_,1\_],[1\_,0\_,1\_,1\_]];\\
		t:=[[0\_,1\_,1\_,0\_],[0\_,1\_,1\_,1\_],[1\_,1\_,1\_,1\_],[0\_,0\_,1\_,1\_]];\\
		a:=[[1\_,0\_,1\_,1\_],[0\_,1\_,1\_,1\_],[0\_,0\_,1\_,0\_],[0\_,0\_,0\_,1\_]];\\
		b:=[[1\_,1\_,1\_,0\_],[1\_,0\_,0\_,0\_],[0\_,0\_,1\_,0\_],[1\_,0\_,1\_,1\_]];\\
		c:=[[0\_,1\_,1\_,1\_],[1\_,0\_,1\_,1\_],[0\_,0\_,1\_,0\_],[0\_,0\_,0\_,1\_]];\\
		d:=[[1\_,0\_,1\_,0\_],[0\_,1\_,1\_,0\_],[0\_,0\_,1\_,0\_],[0\_,0\_,0\_,1\_]];\\
		e:=[[0\_,1\_,0\_,1\_],[0\_,1\_,0\_,0\_],[0\_,0\_,1\_,0\_],[1\_,1\_,0\_,0\_]];\\
		f:=[[1\_,0\_,0\_,0\_],[1\_,0\_,1\_,1\_],[0\_,0\_,1\_,0\_],[1\_,1\_,1\_,0\_]];\\
		H\_N:=Group([s,t]);\\
		H:=Group([a,b,c,d,e,f]);
	}
\end{itemize} 
Next, compute the quotient $H_N/H$ by running \texttt{H\_N\_mod\_H:=H\_N/H;}. When computing the quotient $H_N/H$, GAP automatically verifies that $H$ is a normal subgroup of $H_N$. We can check that the subgroups of interest have the desired structure by calling \texttt{StructureDescription(K)}, where \texttt{K} is one of the groups: \texttt{H\_N}, \texttt{H}, and \texttt{H\_N\_mod\_H}. The only thing left to check, for the case $N=4$, is that the representatives (in $H_N$) of the generators of $H_N/H$ can be chosen to be $x$ and $y$. This can be accomplished as follows.
\texttt{\\x:=[[0\_,1\_,0\_,0\_],[0\_,1\_,1\_,1\_],[1\_,1\_,1\_,1\_],[0\_,0\_,0\_,1\_]];\\
	y:=[[0\_,1\_,0\_,0\_],[0\_,0\_,0\_,1\_],[1\_,1\_,0\_,1\_],[0\_,1\_,1\_,1\_]];\\
	x\^{}3 in H; y\^{}2 in H; y*x*y*x in H;\\
	x in H; x\^{}2 in H; y in H; y*x in H; y*x\^{}2 in H;
}

\subsubsection{Computing the representation $\pi_N'$}
In the next step, we induce $\pi_\rankSL''$ from $H$ to $\SL_\rankSL(\ZZ_{p_N})$ to get $\pi_\rankSL'$. This is done essentially in the \texttt{ind\_H\_to\_G} function, although, for the sake of legibility of the script, we wrapped it in an auxiliary function called \texttt{ind\_rep\_dict} which is invoked directly in the main script.

\subsubsection{Evaluating representations on the Laplacians}
The matrices $\pi_{\rankSL}(\Xi_\rankSL)$ are computed directly from $\pi_\rankSL'$ and $\Xi_\rankSL$ with the \texttt{representing\_matrix} function. For each group ring entry of $\Xi_\rankSL$, we project its supports to $\SL_{\rankSL}(\ZZ_{p_\rankSL})$, apply $\pi_\rankSL'$ to these projections and sum these values with coefficients to get the block entry corresponding to the considered group ring entry of $\Xi_\rankSL$. 

\subsubsection{Checking singularity of Laplacians}
Finally, we compute the corank of the matrices computed in the previous step;  these correspond to $\RR$-dimensions of  $H^2(\SL_3(\ZZ),\pi_3))$ and $H^3(\SL_4(\ZZ),\pi_4)$, respectively.

\subsection{Ensuring rigour of computations}
\label{sec_rigour}
To ensure rigour of the computations, all of them are done over integer or rational types and we use exact determinant and rank functions provided by the \texttt{LinearAlgebraX} package \cite{LinearAlgebraX}. 

\section{Replication of the results}
To replicate the computations justifying \cref{theorem:nontrivial_cohom}, we refer the reader to the \texttt{README.md} file in the Zenodo repository \cite{HigherTSLN2024}. All the replication details are included there as well.

\subsection*{Acknowledgements}
We would like to thank Uri Bader and Roman Sauer for helpful conversations and clarifications regarding the results of \cite{BaderSauer2023}.
BB would like to thank Ric Wade for inviting him to a visit in Oxford where this collaboration started.
He would also like to thank Dan Yasaki for his patient explanations on Voronoi tessellations at the BIRS workshop Cohomology of Arithmetic Groups: Duality, Stability, and Computations (21w5011), organised in 2021 by Jeremy Miller and Jennifer Wilson.

This work has received funding from the European Research Council (ERC) under the European Union's Horizon 2020 research and innovation programme (Grant agreement No. 850930).
It was partially supported by the Deutsche Forschungsgemeinschaft (DFG, German Research Foundation) -- Project-ID 427320536 -- SFB 1442, as well as by Germany’s Excellence Strategy EXC 2044 -- 390685587, Mathematics Münster: Dynamics–Geometry–Structure. The fourth author was supportued by National Science Centre, Poland SONATINA 6 grant \emph{Algebraic spectral gaps in group cohomology} (grant agreement no.
2022/44/C/ST1/00037).

For the purpose of Open Access, the authors have applied a CC BY public copyright licence to any Author Accepted Manuscript (AAM) version arising from this submission.

\bibliographystyle{halpha}
\bibliography{refs.bib}

\end{document}